\documentclass[11pt, leqno,twoside]{article}
\usepackage{amssymb}
\usepackage{amsmath}
\usepackage{amsthm}
\usepackage{graphicx}
\usepackage{cite}
\usepackage[pagebackref,colorlinks,citecolor=blue,linkcolor=blue]{hyperref}

\usepackage{amsmath}
\usepackage{amsfonts}
\usepackage{amssymb}
\usepackage{amsmath,bbm,amssymb,amsxtra}
\usepackage{mathrsfs}
\usepackage{enumerate}
\usepackage{caption}
\allowdisplaybreaks
\usepackage{epsfig}
\usepackage{ae}

  \def\Xint#1{\mathchoice
    {\XXint\displaystyle\textstyle{#1}}%
    {\XXint\textstyle\scriptstyle{#1}}%
    {\XXint\scriptstyle\scriptscriptstyle{#1}}%
    {\XXint\scriptscriptstyle\scriptscriptstyle{#1}}%
    \!\int}

    \def\XXint#1#2#3{{\setbox0=\hbox{$#1{#2#3}{\int}$ }
    \vcenter{\hbox{$#2#3$ }}\kern-.6\wd0}}
    \def\dashint{\Xint-}

\newcommand{\Besov}{{B^{\theta}_{p, p}(\partial X)}}
\newcommand{\dBesov}{{\dot{ B}^{\theta}_{p}(\partial X)}}

\newcommand{\obesov}{{\mathcal B^{\theta,\lambda_2}_{\Phi}(\partial X)}}
\newcommand{\dobesov}{{\dot{\mathcal B}^{\theta, \lambda_2}_{\Phi}(\partial X)}}
\newcommand{\dobesove}{{\dot{\mathcal B}^{\theta, \lambda_2}_{\Phi}}}
\newcommand{\besov}{{\mathcal B^{\theta}_{p}(\partial X)}}

\newcommand{\lbesov}{{\mathcal B^{\theta, \lambda}_{p}(\partial X)}}
\newcommand{\dbesov}{{\dot{\mathcal B}^{\theta}_{p}(\partial X)}}
\newcommand{\dbesove}{{\dot{\mathcal B}^{\theta}_{p}}}
\newcommand{\dlbesov}{{\dot{\mathcal B}^{\theta, \lambda}_{p}(\partial X)}}
\newcommand{\dlbesove}{{\dot{\mathcal B}^{\theta, \lambda}_{p}}}
\newcommand{\lpb}{{L^p(\partial X)}}
\newcommand{\bx}{{{\partial X}}}
\newcommand{\mul}{{{\mu_\lambda}}}
\newcommand{\mull}{{{\mu_{\lambda_2}}}}

\newcommand{\Tr}{{\rm Tr}\,}

\newcommand{\Id}{{\rm Id}\,}

\newcommand{\N}{{\mathbb N}}

\newcommand{\dyadic}{\mathscr{Q}}
\newcommand{\diam}{\text{\rm\,diam}}

\newcommand{\real}{{\mathbb R}}

\newcommand{\rarrow}{\rightarrow}
\newtheorem{thm}{Theorem}[section]
\newtheorem{lem}[thm]{Lemma}
\newtheorem{prop}[thm]{Proposition}

\newtheorem{cor}[thm]{Corollary}
\newtheorem{defn}[thm]{Definition}

\numberwithin{equation}{section}

\begin{document}
\title{\Large\bf Characterization of trace spaces on regular trees\\ via dyadic norms
\footnotetext{\hspace{-0.35cm}
$2010$ {\it Mathematics Subject classfication}: 46E35, 30L05
\endgraf{{\it Key words and phases}: regular tree, trace space, dyadic norm, Orlicz-Sobolev space}
 \endgraf{The author has been supported by the Academy of Finland grant 323960.}
}}
\author{Zhuang Wang}
\date{ }
\maketitle
%\tableofcontents
\begin{abstract}
In this paper, we study the traces of Orlicz-Sobolev spaces on a regular rooted tree. 
After giving a dyadic decomposition of the boundary of the regular tree, we present a characterization on the trace spaces of those first order Orlicz-Sobolev spaces whose Young function is of the form $t^p\log^\lambda(e+t)$, based on integral averages on dyadic elements of the dyadic decomposition. 
 
\end{abstract}

\section{Introduction}
The problem of the  characterization of the trace spaces (on the boundary of a domain) of Sobolev spaces has a long history. It was first studied in the Euclidean setting by Gagliardo \cite{Ga}, who proved that the trace operator $T: W^{1, p}(\real^{n+1}_+)\rarrow B^{1-1/p}_{p,p}(\real^n)$, where $B^{1-1/p}_{p,p}(\real^n)$ stands for the classical Besov space, is linear and bounded for every $p>1$ and that there exists a bounded linear extension operator that acts as a right inverse of $T$. Moreover, he proved that the trace operator $T: W^{1, 1}(\real^{n+1}_+)\rarrow L^1(\real^n)$ is a bounded linear surjective operator with a non-linear right inverse.  Peetre \cite{Pe} showed that one can not find a bounded linear extension operator  that acts as a right inverse of $T: W^{1, 1}(\real^{n+1}_+)\rarrow L^1(\real^n)$. We refer to the seminal monographs by Peetre \cite{Pe79} and Triebel \cite{T83,T01} for extensive treatments of the Besov spaces and related smoothness spaces.
 In potential theory, certain types of Dirichlet problem are guaranteed to have solutions when the boundary data belongs to a trace space corresponding to the Sobolev class on the domain. In the Euclidean setting, we refer to \cite{Ar,Li,MirRus,SlBa,Tyu1,Tyu2} for more information on the traces of (weighted) Sobolev spaces and \cite{DhKa,DhKa2,Fou,Lac72,Pal79,C10,Lac,Pal}  for results on traces of (weighted) Orlicz-Sobolev spaces. 

Analysis on metric measure spaces has recently been under active study, e.g., \cite{BB11,BBS03,H03,HP,H01,HK98,HKST}. Especially the trace theory in the metric setting has been under development.  Mal\'y \cite{Ma} proved that the trace space of the Newtonian space $N^{1, p}(\Omega)$ is the Besov space $B^{1-\theta/p}_{p, p}(\partial\Omega)$ provided that $\Omega$ is a John domain for $p>1$ (uniform domain for $p\geq 1$) that admits a $p$-Poincar\'e inequality and whose boundary $\partial \Omega$ is endowed with a codimensional-$\theta$ Ahlfors regular measure with $\theta<p$. We also refer to the paper \cite{SS} for studies on the traces of  Haj\l asz-Sobolev functions to porous Ahlfors regular closed subsets via a method  based on hyperbolic fillings of a metric space, see \cite{BS18,So}.   For the trace result of BV (bounded variation) functions, we refer to \cite{LS18,MSS,LXZ19}.

%Much of the recent development of analysis in metric measure spaces has tended to focus on the metric measure spaces that are fractals with minimal connectedness property, such as the Sierpi\'nski gasket, see \cite{a} and the references therein. 
The recent paper \cite{BBGS} dealt with  geometric analysis on Cantor-type sets which are uniformly perfect totally disconnected metric measure spaces, including various types of Cantor sets. Cantor sets embedded in Euclidean spaces support a fractional Sobolev space theory based on Besov spaces. Indeed, suitable Besov functions on such a set are traces of the classical Sobolev functions on the ambient Euclidean spaces, see Jonsson-Wallin \cite{JoWa80,JoWa}. The paper \cite{BBGS,KW} established similar trace and extension theorems for Sobolev and Besov spaces on regular trees and their Cantor-type boundaries.  Indeed, for a regular $K$-ary tree $X$ with $K\geq 2$ and its Cantor-type boundary $\bx$, if we give the uniformizing metric (see \eqref{metric})
$$d_X(x, y)=\int_{[x, y]} e^{-\epsilon|z|}\, d\,|z|$$
and the weighted measure (see \eqref{def-measure} )
\begin{equation}\label{def-measure0}
d\mul(x)=e^{-\beta|x|}(|x|+C)^\lambda\, d\,|x|
\end{equation}
on $X$, then the Besov space $\lbesov$ in Definition \ref{besov-3} below is exactly the trace of the Newton-Sobolev space $N^{1, p}(X, \mul)$ defined in Section \ref{Newtonian-space}, see \cite[Theorem 1.1]{KW} and \cite[Theorem 6.5]{BBGS}. Here the smoothness exponent of the Besov space is
$$\theta=1-\frac{\beta/\epsilon-Q}{p},\ \ 0<\theta<1,$$
where $Q=\log K/\epsilon$  is the Hausdorff dimension of the Cantor-type boundary and $\beta/\epsilon-Q$ is a ``codimension'' determined by the uniformizing metric $d_X$ and the measure $\mu$ on the tree.

In Euclidean spaces, the classical Besov norm is equivalent to a dyadic norm, and the trace spaces of the Sobolev spaces can be characterized by the Besov spaces defined via dyadic norms,  see e.g. \cite[Theorem 1.1]{KTW}. Inspired by this, we give a dyadic decomposition of the boundary $\bx$ and define a Besov space $\besov$ on the boundary $\bx$ by using a dyadic norm, see Section \ref{s-besov} and Definition \ref{besov-1}. We show in Proposition \ref{norm-equiv} that the dyadic Besov spaces $\besov$ coincide with the Besov space $B^{\theta}_{p, p}(\bx)$ and the Haj\l asz-Besov space $N^{\theta}_{p, p}(\bx)$, see Definition \ref{besov-0} and Definition \ref{H-besov} for definitions of $B^{\theta}_{p, p}(\bx)$ and $N^{\theta}_{p, p}(\bx)$. We refer to \cite{BBGS,GKZ11,K19,GKS10,KYZ10,KYZ11} for more information about
Besov spaces $B^{\theta}_{p, p}(\cdot)$ and Haj\l asz-Besov spaces $N^{\theta}_{p, p}(\cdot)$ on metric measure spaces.

By relying on dyadic norms,  we define the Orlicz-Besov space $\obesov$, $\lambda_2\in \mathbb R$ for the Young function  $\Phi(t)=t^p\log^{\lambda_1}(e+t)$ with $p >1, \lambda_1\in\real$ or $p=1, \lambda_1\geq 0$, see Definition \ref{besov-2}. Our first result shows that the Orlicz-Besov space $\obesov$ is the trace space of the Orlicz-Sobolev space $N^{1, \Phi}(X, \mu_{\lambda_2})$ defined in  Section \ref{Newtonian-space}.

\begin{thm}\label{th1}
Let $X$ be a $K$-ary tree with $K\geq 2$ and let $\Phi(t)=t^p\log^{\lambda_1}(e+t)$ with $p >1, \lambda_1\in\real$ or  $p=1, \lambda_1\geq 0$. Fix $\lambda_2\in \real$ and let $\mu_{\lambda_2}$ be the weighted measure given by \eqref{def-measure0}.  Assume that  $p>(\beta-\log K)/\epsilon>0$. Then the trace space of $N^{1, \Phi}(X, \mu_{\lambda_2})$ is the space $\obesov$ where $\theta=1-(\beta-\log K)/\epsilon p$.
\end{thm}
In this paper, for given Banach spaces $\mathbb X(\bx)$ and $\mathbb Y(X)$, we call  the space $\mathbb X(\bx)$  a trace space of $\mathbb Y(X)$ if and only if there exist a bounded linear operator $T: \mathbb Y(X)\rightarrow \mathbb X(\bx)$ and  a bounded linear extension operator  $E: \mathbb X(\bx)\rightarrow \mathbb Y(X)$ such that $T\circ E=\Id$ on the space $\mathbb X(\bx)$.

Our next result identifies the Orlicz-Besov space $\obesov$ as the Besov space $\lbesov$.

\begin{prop}\label{prop1.2} 
Let $\lambda, \lambda_1, \lambda_2\in \real$.
Let $\Phi(t)=t^p\log^{\lambda_1}(e+t)$ with $p >1, \lambda_1\in\real$ or  $p=1, \lambda_1\geq 0$. Assume that $\lambda_1+\lambda_2=\lambda$. Then the Banach spaces $\lbesov$ and $\obesov$ coincide, i.e., $\lbesov=\obesov$.
\end{prop}

By combining Theorem \ref{th1} and Proposition \ref{prop1.2}, we obtain the following result.

\begin{cor}\label{cor1.3}
Let $X$ be a $K$-ary tree with $K\geq 2$. Let $\lambda, \lambda_1, \lambda_2\in \real$.
 Assume that $p>(\beta-\log K)/\epsilon>0$ and let $\theta=1-(\beta-\log K)/\epsilon p$. Let $\Phi(t)=t^p\log^{\lambda_1}(e+t)$ with $p >1, \lambda_1\in\real$ or  $p=1, \lambda_1\geq 0$. Then the Besov-type space $\lbesov$ is the trace space of $N^{1, \Phi}(X, \mu_{\lambda_2})$ whenever $\lambda_1+\lambda_2=\lambda$.
\end{cor}

When $\lambda_1=0$ and $\lambda_2=\lambda$, the above result coincides with \cite[Theorem 1.1]{KW}, which states that the Besov-type space $\lbesov$ is the trace space of $N^{1, p}(X, \mu_{\lambda})$ for a suitable $\theta$. The above result shows that the Besov-type space $\lbesov$ is not only the trace space of $N^{1, p}(X, \mu_{\lambda})$ but  actually the trace space of all these Orlicz-Sobolev spaces $N^{1, \Phi}(X, \mu_{\lambda_2})$ (including $N^{1, p}(X, \mu_{\lambda})$) for suitable $\theta, \lambda_2$ and $\Phi$. It may be worth to point out here that these Orlicz-Sobolev spaces $N^{1, \Phi}(X, \mu_{\lambda_2})$ are different from each other.

The paper is organized as follows. In Section \ref{s2}, we give all the necessary preliminaries. More precisely, we introduce  regular trees in Section \ref{s-regular} and we consider a doubling property of the measure $\mu$ on a regular tree $X$ and the Ahlfors regularity of  its boundary $\bx$. The definition of Young functions is given in Section \ref{s-Young}. We introduce the Newtonian and Orlicz-Sobolev spaces on $X$ and the Besov-type spaces on $\bx$ in Section \ref{Newtonian-space} and Section \ref{s-besov}, respectively. In Section \ref{proofs}, we give the  proofs of Theorem \ref{th1} and Proposition \ref{prop1.2}. 

In what follows, the letter $C$ denotes a constant that may change at different occurrences. The notation $A\approx B$ means that there is a constant $C$ such that $1/C\cdot A\leq B\leq C\cdot A$. The notation $A\lesssim B$ ($A\gtrsim B$) means that there is a constant $C$ such that  $A\leq  C\cdot B$ ($A\geq C\cdot B$).

%%%%%%%%%%%%%%%%%%%%%%%%%%%%%%%%%%%%%%%%%%%%%%%%%%%%%%%%%%%%%

\section{Preliminaries}\label{s2}
\subsection{Regular trees and their boundaries}\label{s-regular}
A {\it graph} $G$ is a pair $(V, E)$, where $V$ is a set of vertices and $E$ is a set of edges.  Given vertices $x, y\in V$ are neighbors if $x$ is connected to $y$ by an edge. The number of the neighbors of a vertex $x$ is referred to as the degree of $x$.  A {\it tree} $G$ is a connected graph without cycles. 

Let us fix a vertex that we refer to by $0$. The neighbors of $0$ will be called children of $0$ and $0$ is called their mother. If $x$ is one of the children of $0$, then the neighbors of $x$ different from $0$ are called children of $x$ and we say that $x$ is their mother. We continue in the obvious manner to define the children and the mother for all $y\not=0$.
We then call $G$ a rooted tree with root $0$ and say that $G$ is $K$-regular if additionally each vertex has precisely $K$ children.

Let $G$ be a $K$-regular tree with a set of vertices $V$ and  a set of edges $E$ for some $K\geq 1$. For simplicity of notation, we let $X=V\cup E$ and call it a $K$-regular tree.  We consider each edge as a geodesic of length one.
For $x\in X$, let $|x|$ be the length of the geodesic from $0$ to $x$, where we consider each edge to be an isometric copy of the unit interval. The geodesic connecting  $x, y\in V$ is unique. We refer to it by $[x, y]$, and to its length by $|x-y|$.  We write $x\leq y$ if $x\in [0, y]$. Then $|x-y|=|y|-|x|$. We say that a vertex $y\not=x$ is a descendant of the vertex $x$ if $x\leq y$.

Towards defining the metric of $X$, let $\epsilon >0$, and set 
\begin{equation}\label{metric}
d_X(x, y)=\int_{[x, y]} e^{-\epsilon|z|}\, d\,|z|.
\end{equation}
Here $d\,|z|$ is the natural measure that gives each edge Lebesgue measure $1$; recall that  each edge is an isometric copy of the unit interval. Notice that $\diam X=2/\epsilon$ if $X$ is a $K$-ary tree with $K\geq 2$.  

The boundary $\bx$ of a tree $X$ is obtained by completing $X$ with respect to the metric $d_X$.  An element $\xi\in \bx$ can be identified with an infinite geodesic  starting at the root $0$. Equivalently we employ the labeling $\xi=0x_1x_2\cdots$, where $x_i$ is a vertex in $X$ with $|x_i|=i$, and $x_{i+1}$ is a child of $x_i$. The extension of the metric to $\bx$ can be realized in the following manner.  Given $\xi, \zeta\in \bx$, pick an infinite geodesic $[\xi, \zeta]$ connecting $\xi$ and $\zeta$. Then $d_X(\xi, \zeta)$ is the length of the  geodesic $[\xi, \zeta]$. Indeed, if $\xi=0x_1x_2\cdots$ and $\zeta=0y_1y_2\cdots$, let $k$ be the integer with $x_k=y_k$ and $x_{k+1}\not=y_{k+1}$. Then 
$$d_X(\xi, \zeta)=2\int_{k}^{+\infty} e^{-\epsilon t}\, d t=\frac 2\epsilon e^{-\epsilon k}.$$
 For more details, see \cite{BBGS,BHK01,BH99}.
For clarity, we use $\xi, \zeta, \omega$ to denote points in $\bx$ and $x, y, z$  points in $X$. 

On the regular $K$-ary tree $X$, we use the weighted measure $\mul$ introduced in \cite[Section 2.2]{KW}, defined by
\begin{equation}\label{def-measure}
d\mul(x)=e^{-\beta|x|}(|x|+C)^\lambda\, d\,|x|,
\end{equation}
 where $\beta>\log K$, $\lambda\in \mathbb R$ and $C\geq\max\{2|\lambda|/(\beta-\log K), 2(\log 4)/\epsilon\}$. For $\lambda=0$, this is the measure used in \cite{BBGS}. 

The following proposition gives the doubling property of the measure $\mul$, see \cite[Corollary 2.9]{KW}.

\begin{prop}
For any $\lambda\in \mathbb R$, 
the measure $\mu_\lambda$ is doubling, i.e., $\mul(B(x, 2r))\lesssim \mul(B(x, r))$.
\end{prop}

The result in \cite[Lemma 5.2]{BBGS} shows that the boundary $\bx$ of the regular $K$-ary tree $X$ is Ahlfors regular with the regularity exponent depending only on $K$ and on the metric density exponent $\epsilon$ of the tree. 
\begin{prop}\label{Ahlfor-boundary}
The boundary $\bx$ is an Ahlfors $Q$-regular space with Hausdorff dimension 
$$Q=\frac{\log K}{\epsilon}.$$
\end{prop}
Hence $\bx$ is equipped with an Ahlfors $Q$-regular measure $\nu$:
$$\nu(B_{\bx}(\xi, r)) \approx r^Q=r^{\log K/\epsilon},$$
for any $\xi\in \bx$ and $0<r\leq \diam \bx.$

Throughout the paper we assume that $1\leq p<+\infty$ and that $X$ is a $K$-ary tree with $K\geq 2$.
 
\subsection{Young functions and Orlicz spaces}\label{s-Young}
In the standard definition of an Orlicz space, the function $t^p$ of an $L^p$-space is replaced with a more general convex function, a Young function. We recall the definition of a Young function. We refer  to \cite[section 2.2]{T04} and \cite{RR91} for more details about Young functions and we also warn the reader of slight differences between the definitions in various references.

A function $\Phi: [0, \infty)\rarrow [0, \infty)$ is a {\it Young function} if 
it is a continuous, increasing and convex function satisfying  $\Phi(0)=0$,
$$\lim_{t\rarrow 0+}\frac{\Phi(t)}{t}=0\ \ \text{and}\ \ \  \lim_{t\rarrow +\infty}\frac{\Phi(t)}{t}=+\infty.$$
A  Young function $\Phi$ can be expressed as
$$\Phi(t)=\int_0^t \phi(s)\, ds,$$
where $\phi: [0, \infty)\rarrow [0, \infty)$ is an increasing, right-continuous function with $\phi(0)=0$ and $\underset{t \rightarrow + \infty}{\lim}\phi(t)=+\infty.$ 

A Young function $\Phi$ is said to satisfy the $\Delta_2 -$condition if there is a constant $C_\Phi>0$, called a {\it doubling constant} of $\Phi$,  such that 
\begin{equation*}
\Phi(2 t) \le C_\Phi \Phi(t),\ \forall\ \  t \ge 0. 
\end{equation*}

If Young function $\Phi$ satisfies the $\Delta_2 -$condition, then for any constant $c>0$, there exist $c_1, c_2>0$ such that 
$$c_1\Phi(t)\leq \Phi(ct)\leq c_2\Phi(t)\ \ \ \ {\rm for\  all} \ \ \ t\geq 0,$$ 
where $c_1$ and $c_2$ depend only on $c$ and the doubling constant $C_\Phi$.  Therefore, we obtain that if $A\approx B$, then $  \Phi(A)\approx \Phi(B)$.  This property will be used frequently in the rest of this paper.

Let $\Phi_1, \Phi_2$ be two Young functions. If there exist two constants $k>0$ and $C\geq 0$ such that 
$$\Phi_1(t)\leq \Phi_2(kt)\ \ \ \ {\rm for} \ \ \ t\geq C,$$
we write 
$$\Phi_1\prec\Phi_2.$$

The function $\Phi(t)=t^p\log^{\lambda}(e+t)$ with $p >1, \lambda\in\real$ or $p=1, \lambda\geq 0$ is a Young function and it satisfies the $\Delta_2 -$condition. Moreover, it also satisfies that
\begin{equation}\label{eq-add}
t^{\max\{p-\delta, 1\}}\prec \Phi(t)\prec t^{p+\delta}
\end{equation}
for any $\delta>0$.

Let $\Phi$ be a Young function. Then the {\it Orlicz space} $L^\Phi(X)$ is defined by setting
$$L^{\Phi}(X, \mul)=\left\{u: X\rarrow \real: u\ {\rm measurable,}\ \int_X\Phi(\alpha|u|)\, d\mul<+\infty\ {\rm for \  some}\ \alpha>0\right\}.$$
As in the theory of $L^p$-spaces, the elements in $L^\Phi(X, \mul)$ are actually equivalence classes consisting of functions that differ only on a set of measure zero. The Orlicz space $L^\Phi(X, \mul)$ is a vector space  and, equipped with the {\it Luxemburg norm}
$$\|u\|_{L^\Phi(X, \mul)}=\inf\left\{k>0: \int_{X} \Phi(|u|/k)\, d\mul\leq 1\right\},$$
a Banach space, see \cite[Theorem 3.3.10]{RR91}. If $\Phi(t)=t^p$ with $p\geq 1$, then $L^{\Phi}(X, \mul)=L^p(X, \mul)$. We refer to \cite{Orl60,RR91,T04}  for more detailed discussions and properties of Orlicz spaces. 

\subsection{Newtonian spaces and Orlicz-Sobolev spaces on $X$}\label{Newtonian-space}
We call a Borel function $g: X\rarrow [0, \infty]$  an {\it upper gradient} of $u\in L^{1}_{\rm loc}(X, \mul)$ if 
\begin{equation}\label{gradient}|u(z)-u(y)|\leq \int_{\gamma} g\, ds_X\end{equation}
whenever $z, y\in X$ and $\gamma$ is the geodesic from $z$ to $y$, where $d s_X$ denotes the arc length measure with respect to the metric $d_X$. Since any rectifiable curve with end points $z$ and $y$ in our tree contains the corresponding geodesic,  the above definition  is equivalent to the usual  definition which requires that inequality \eqref{gradient} holds for all rectifiable curves with end points $z$ and $y$. See \cite{BB11,H03,HK98,HKST,N00} for a more detailed discussion on upper gradients.

The {\it Newtonian space} $N^{1, p}(X, \mul)$, $1\leq p<\infty$, is  the collection of all functions $u$ for which the norm of $u$ defined as
$$\|u\|_{N^{1, p}(X, \mul)}:= \left(\int_X |u|^p\, d\mul+\inf_g \int_X g^p\, d\mul\right)^{1/p}$$
is finite. Here  the infimum is taken over all upper gradients of $u$. 

For any Young function $\Phi$, the {\it Orlicz-Sobolev space} $N^{1, \Phi}(X, \mul)$  is defined as the collection of all functions $u$ for which the norm of $u$ defined as
$$\|u\|_{N^{1, \Phi}(X, \mul)}=\|u\|_{L^{\Phi}(X, \mul)}+\inf_g \|g\|_{L^{\Phi}(X, \mul)}$$
is finite,
where the infimum is taken over all upper gradients of $u$. 

 For the Young function $\Phi(t)=t^p$, $1\leq p<\infty$, the Orlicz-Sobolev space $N^{1, \Phi}(X, \mul)$ is exactly the Newtonian space $N^{1, p}(X, \mul)$. We refer to \cite{T04} for further results on Orlicz-Sobolev spaces on metric measure spaces. If $u\in N^{1, p}(X, \mul)$ ($u\in N^{1, \Phi}(X, \mul)$ with $\Phi$ doubling), then it has a minimal $p$-weak upper gradient ($\Phi$-weak upper gradient) $g_u$, which in our case is an upper gradient. The minimal upper gradient is minimal in the sense that  if $g\in L^p(X, \mul)$ ($g\in L^\Phi(X, \mul)$) is any upper gradient of $u$, then $g_u\leq g$ a.e. We refer the interested reader to \cite[Theorem 7.16]{H03} ($p\geq 1$) and  \cite [Corollary 6.9]{T04}($\Phi$ doubling) for proofs of the existence of such a minimal upper gradient.

\subsection{Besov-type spaces on $\bx$}\label{s-besov}
Towards the definition of our Besov-type spaces, we recall a definition from \cite{BBGS}.
\begin{defn}\label{besov-0}\rm
For $0<\theta<1$ and $p\geq 1$, The Besov space $\Besov$ consists of all functions $f\in L^p(\bx)$ for which the seminorm  $\|f\|_{\dBesov}$ defined as
$$\|f\|^p_{\dBesov}:=\int_{\bx}\int_{\bx}\frac{|f(\zeta)|-f(\xi)|^p}{d_X(\zeta, \xi)^{\theta p}\nu(B(\zeta, d_X(\zeta, \xi)))}d\nu(\xi)\, d\nu(\zeta)$$
is finite. The corresponding norm for $\Besov$ is 
$$\|f\|_{\Besov}:=\|f\|_{\lpb}+\|f\|_{\dBesov}.$$
\end{defn}

We base our definition on a dyadic decomposition on the boundary $\bx$ of the $K$-ary tree $X$, see also \cite[Section 2.4]{KW}.
 Let $V_n=\{x_j^n:  j=1, 2, \cdots, K^n\}$ be the set of all $n$-level vertices of the tree $X$ for each $n\in \N$, where a vertex $x$ is of {\it $n$-level} if $|x|=n$. Then 
$$V=\bigcup_{n\in \N} V_n.$$ 
Given a vertex $x\in V$,  set 
$$I_x:=\{\xi\in \bx: \text{the geodesic $[0, \xi)$ passes through $x$}\}.$$ 
Let $\dyadic=\{I_x: x\in V\}$ and $\dyadic_n=\{I_x: x\in V_n\}$ for each $n\in \N$. Then $\dyadic_0=\{\partial X\}$ and our dyadic decomposition $\dyadic$ satisfies
$$\dyadic =\bigcup_{n\in \N} \dyadic_n.$$
 Given $I\in \dyadic_{n}$, there is a unique element $\widehat I$ in $\dyadic_{n-1}$ such that $I\subset \widehat I$. If $I=I_x$ for some $x\in V_{n}$, then $\widehat I=I_y$ where  $y$ is the unique mother of $x$ in the tree $X$. Hence the structure of the dyadic decomposition of $\partial X$ is uniquely determined by the structure of the $K$-ary  tree $X$.

We recall a definition from \cite{KW}.
\begin{defn}\label{besov-3}\rm
For $0\leq\theta<1$, $p\geq 1$ and $\lambda\in \mathbb R$,  the Besov-type space $\lbesov$ consists of all functions $f\in L^p(\bx)$ for which the $\dlbesove$-dyadic energy of $f$ defined as
$$\|f\|^p_{\dlbesov}:=\sum_{n=1}^{\infty} e^{\epsilon n\theta p}n^\lambda \sum_{I\in \dyadic_n} \nu(I)\left|f_{I}-f_{ \widehat I}\right|^p$$
  is finite. The norm on $\lbesov$ is 
$$\|f\|_{\lbesov}:=\|f\|_{\lpb}+\|f\|_{\dlbesov}.$$
\end{defn}

The measure $\nu$ above is  the Ahlfors regular measure given by Proposition \ref{Ahlfor-boundary} and $f_I := \dashint_I f\, d\nu =\frac{1}{\nu(I)} \int_I f\, d\nu$ is the usual mean value.
 
 \begin{defn}\label{besov-1}\rm
For $0<\theta<1$ and $p\geq 1$, The Besov space $\besov$ consists of all the functions $f\in L^p(\bx)$ for which the $\dbesove$-dyadic energy of $f$ defined as
$$\|f\|^p_{\dbesov}:=\sum_{n=1}^{\infty} e^{\epsilon n\theta p}\sum_{I\in \dyadic_n} \nu(I)\left|f_{I}-f_{ \widehat I}\right|^p$$
is finite. The norm of $\besov$ is 
$$\|f\|_{\besov}:=\|f\|_{\lpb}+\|f\|_{\dbesov}.$$
\end{defn}
The Besov-type spaces $\lbesov$ and $\besov$ were first introduced in \cite{KW}. Notice that $\besov$ coincides with $\lbesov$ when $\lambda=0$.
Next we introduce the {\it Haj\l asz-Besov spaces} $N^{\theta
}_{p,p}(\bx)$ on the boundary $\bx$.
\begin{defn}\rm\label{H-besov}
(i) Let $0<\theta<\infty$ and let $u$ be a measurable function on $\bx$. A sequence of nonnegative measurable functions, $\vec{g}=\{g_k\}_{k\in \mathbb Z}$, is called a {\it fractional $\theta$-Haj\l asz gradient} of $u$ if there exists $Z\subset \bx$ with $\nu(Z)=0$ such that for all $k\in \mathbb Z$ and $\zeta, \xi\in \bx\setminus Z$ satisfying $2^{-k-1}\leq d_X(\zeta, \xi)<2^{-k}$,
\[
|u(\zeta)-u(\xi)|\leq [ d_X(\zeta, \xi)]^{\theta} [g_k(\zeta)+g_k(\xi)].
\]
Denote by $\mathbb D^\theta(u)$ the {\it collection of all fractional $\theta$-Haj\l asz gradients of $u$}.

(ii) Let $0<\theta<\infty$ and $0<p<\infty$. The {\it Haj\l asz-Besov space} $N^{\theta}_{p, p}(\bx)$ consists of  all functions $u\in L^p(\bx)$ for which the seminorm $\|u\|_{\dot N^{\theta}_{p, p}(\bx)}$ defined as
\[
\|u\|_{\dot N^{\theta}_{p, p}(\bx)}:= \inf_{\vec{g}\in \mathbb D^{\theta}(u)}\|(\|g_k\|_{L^p(\bx)})_{k\in \mathbb Z}\|_{l^p}=\inf_{\vec{g}\in \mathbb D^{\theta}(u)}\left(\sum_{k\in \mathbb Z} \int_{\bx} [g_k(\xi)]^p\, d\nu(\xi) \right)^{1/p}
\]
is finite. The norm of $N^{\theta}_{p, p}(\bx)$ is  
\[\|u\|_{ N^{\theta}_{p, p}(\bx)}:=\|u\|_{L^p(\bx)}+\|u\|_{\dot N^{\theta}_{p, p}(\bx)}.\]
\end{defn}

The following proposition states that these three Besov-type spaces $\besov$, $\Besov$ and $N^{\theta}_{p, p}(\bx)$ coincide with each other.% The proof of this proposition follows by using \cite[Lemma 5.4]{BBGS} and a  modification of the proof of \cite[Proposition A.1]{KTW}. We omit the details.

\begin{prop}\label{norm-equiv}
Let $0<\theta<1$ and $p\geq 1$. For any $f\in L^1_{\rm loc}(\bx)$, we have
$$\|f\|_{\dBesov}\approx \|f\|_{\dbesov}\approx \|f\|_{\dot N^{\theta}_{p, p}(\bx)}.$$
%and 
%$$\|f\|_{\dBesov}\approx \|f\|_{\dbesov}\approx \|f\|_{\dot N^{s}_{p, p}(\bx)}$$
\end{prop}
\begin{proof}
The first part $\|f\|_{\dBesov}\approx \|f\|_{\dbesov}$ follows by \cite[Proposition 2.13]{KW}.

The second part  $\|f\|_{\dBesov}\approx \|f\|_{\dot N^{s}_{p, p}(\bx)}$ is given by \cite[Lemma 5.4]{BBGS} and \cite[Theorem 1.2]{GKZ11}. 
\end{proof}

The dyadic norms give an easy way to introduce  Orlicz-Besov spaces by replacing $t^p$ with some Orlicz function $\Phi(t)$.  
\begin{defn}\label{besov-2}\rm
Let $\Phi$ be the  Young function  $\Phi(t)=t^p\log^{\lambda_1}(e+t)$ with $p >1, \lambda_1\in\real$ or $p=1, \lambda_1\geq 0$. Then the  Orlicz-Besov space $\obesov$ consists of all $f\in L^\Phi(\bx)$ whose norm generally defined as
$$\|f\|_{\obesov}:=\|f\|_{L^\Phi(\bx)}+\inf\left\{k>0: |f/k|_{\dobesov}\leq 1\right\}$$
is finite, where for any $g\in L^1_{\rm loc}(\bx)$, the $\dobesove$-dyadic energy is defined as
$$|g|_{\dobesov}:=\sum_{n=1}^{\infty} e^{\epsilon n(\theta-1) p}n^{\lambda_2}\sum_{I\in \dyadic_n} \nu(I)\Phi\left(\frac{\left|g_{I}-g_{\widehat I}\right|}{e^{-\epsilon n}}\right).$$
\end{defn}
In this paper, we are only interested in  the Young functions in the above definition. Hence in the rest of  this paper, we always assume that the Young function is $\Phi(t)=t^p\log^{\lambda_1}(e+t)$ with $p >1, \lambda_1\in\real$ or $p=1, \lambda_1\geq 0$.

%%%%%%%%%%%%%%%%%%%%%%%%%%%%%%%%%%%%%%%%%%%%%%%%%%%%%%%%
\section{Proofs}\label{proofs}

\subsection{Proof of Theorem \ref{th1}}
\begin{proof}
{\bf Trace Part:}
Let $f\in N^{1, \Phi}(X)$. We follow an idea from \cite{KW} and set
\begin{equation}\label{trace-operator}
\Tr f(\xi):=\tilde f(\xi)=\lim_{[0, \xi)\ni x\rarrow \xi} f(x), \ \ \xi\in \bx
\end{equation}
 provided that the limit  taken along the geodesic ray $[0, \xi)$ exists. We begin by showing that the above limit exists for $\nu$-a.e. $\xi\in \bx$.

Since $g_f$ is  an upper gradient of $f$, 
 it suffices to show that the function $f^*$ defined by setting
\begin{equation}\label{trace-operator1}
{\tilde f}^*(\xi)=|f(0)|+\int_{[0, \xi)} g_f\, ds
\end{equation}
belongs to $L^p(\bx)$, where $[0, \xi)$ is the geodesic ray from $0$ to $\xi$. Indeed, if $\tilde f^*\in L^p(\bx)$, we have $|\tilde f^*|<\infty$ for $\nu$-a.e. $\xi\in \bx$, and hence the limit in $\eqref{trace-operator}$ exists for $\nu$-a.e. $\xi\in\bx$.

Fix $\xi\in \bx$.  
Set $r_j=2e^{-j\epsilon}/\epsilon$ and $x_j=x_j(\xi)$ be the ancestor of $\xi$ with $|x_j|=j$ for $j\in \mathbb N$. Then
% \begin{equation}\label{relation}
% ds\approx e^{(\beta-\epsilon)j}\, d\mu\approx r_j^{1-\beta/\epsilon}\, d\mu \ \ \ {\rm and}\ \ \mu([x_{j}, x_{j+1}])\approx r_j^{\beta/\epsilon}.
% \end{equation}
\begin{equation}\label{relation1}
ds\approx e^{(\beta-\epsilon)j} j^{-\lambda_2}\, d\mull\approx r_j^{1-\beta/\epsilon} j^{-\lambda_2}\, d\mu, \ \ \ \ \mull([x_{j}, x_{j+1}])\approx r_j^{\beta/\epsilon} j^{\lambda_2},
\end{equation}
where $[x_j, x_{j+1}]$ is the edge connecting $x_j=x_j(\xi)$ and $x_{j+1}=x_{j+1}(\xi)$.
Thus
\begin{align}
{\tilde f}^*(\xi)&= |f(0)|+\sum_{j=0}^{+\infty} \int_{[x_{j}, x_{j+1}]} g_f \, ds\notag \\
&\approx |f(0)|+\sum_{j=0}^{+\infty}{r_j^{1-\beta/\epsilon}}j^{-\lambda_2}\int_{[x_{j}, x_{j+1}]} g_f \, d\mull\notag\\
&\approx |f(0)|+\sum_{j=0}^{+\infty}{r_j}\dashint_{[x_{j}, x_{j+1}]} g_f \, d\mull.\label{estimate-f}
\end{align}

Since $\theta=1-(\beta-\log K)/(p\epsilon)>0$,  we may choose $1\leq q<\infty$ such that $\max\{(\beta-\log K)/\epsilon, 1\}<q<p$ if $p>1$ or $q=1=p$. Let $\Psi(t)=t^{p/q}\log^{\lambda/q} (e+t)$. Then $\Psi^q=\Phi$ and $\Psi$ is a doubling Young function.  By the Jensen inequality and the doubling property of $\Psi$, since $\sum_{j=0}^{+\infty} r_j\approx 1$, we have that
\begin{align*}
\Psi({\tilde f}^*(\xi))&\lesssim \Psi (|f(0)|)+\Psi\left(\sum_{j=0}^{+\infty}{r_j}\dashint_{[x_{j}, x_{j+1}]} g_f \, d\mull\right)\\
&\lesssim \Psi (|f(0)|)+\sum_{j=0}^{+\infty} r_j  \dashint_{[x_{j}, x_{j+1}]} \Psi(g_f) \, d\mull.
\end{align*}
Choose $0<\kappa<1-(\beta-\log K)/(q\epsilon)$. If $q>1$, by the H\"older inequality, we obtain the estimate
\begin{align*}
\Phi({\tilde f}^*(\xi))=\Psi({\tilde f}^*(\xi))^q &\lesssim \Phi(|f(0)|)+  \left(\sum_{j=0}^{+\infty} r_j^{\kappa}\, r_j^{(1-\kappa)} \dashint_{[x_{j}, x_{j+1}]} \Psi(g_f) \, d\mull\right)^q\\
&\lesssim \Phi(|f(0)|)+  \sum_{j=0}^{+\infty}r_j^{(1-\kappa)q} \left(\dashint_{[x_{j}, x_{j+1}]} \Psi(g_f) \, d\mull\right)^q\\
& \lesssim \Phi(|f(0)|)+ \sum_{j=0}^{+\infty}  r_j^{q-\kappa q-\beta/\epsilon}j^{-\lambda_2}\int_{[x_{j}, x_{j+1}]} \Phi(g_f) \, d\mull.
\end{align*}
Here the second inequality use the fact
$$\sum_{j=0}^{+\infty} r_j^{kq/(q-1)}\approx 1.$$
If $q=1$, then $\Psi=\Phi$, and hence the H\"older inequality is not needed in the estimate. We conclude that 
\begin{equation*}
\Phi({\tilde f}^*(\xi))\lesssim \Phi(|f(0)|)+ \sum_{j=0}^{+\infty}  r_j^{q-\kappa q-\beta/\epsilon}j^{-\lambda_2}\int_{[x_{j}, x_{j+1}]} \Phi(g_f) \, d\mu.
\end{equation*}
Since $\nu(\bx)\approx 1$, integration of this estimate over $\bx$ together with Fubini's theorem  gives
\begin{align}
\int_\bx \Phi({\tilde f}^*(\xi))\, d\nu&\lesssim \Phi(|f(0)|) +\int_\bx\sum_{j=0}^{+\infty}  r_j^{q-\kappa q-\beta/\epsilon}j^{-\lambda_2}\int_{[x_{j}, x_{j+1}]} \Phi(g_f) \, d\mull\, d\nu(\xi)\notag\\
&=\Phi(|f(0)|) +\int_X \Phi(g_f(x))\int_\bx \sum_{j=0}^{+\infty}  r_j^{q-\kappa q-\beta/\epsilon}j^{-\lambda_2}\chi_{[x_j, x_{j+1}]}(x)\, d\nu(\xi)\, d\mull(x). \label{add-eq-1}
\end{align}
Since $\chi_{[x_j, x_{j+1}]}(x)$ is nonzero only if $j\leq |x|\leq j+1$ and $x<\xi$, our estimate \eqref{add-eq-1} can be reformulated as 
\begin{equation}\label{add-eq-2}
\int_\bx \Phi({\tilde f}^*(\xi))\, d\nu\lesssim \Phi(|f(0)|) +\int_X \Phi(g_f(x)) r_{j(x)}^{q-\kappa q-\beta/\epsilon}j(x)^{-\lambda_2}\nu(E(x))\, d\mu(x),
\end{equation}
where $E(x)=\{\xi\in \bx: x<\xi\}$ and $j(x)$ is the largest integer such that $j(x)\leq |x|$.

By Proposition \ref{Ahlfor-boundary}, we have $\nu(E(x))\approx r_{j(x)}^Q$, since $E(x)=B(\xi, r)$ for any $\xi\in E(x)$ and $r=\diam(E(x))/2\approx e^{-\epsilon j(x)}$, see \cite[Lemma 5.21]{BBGS}.  This together with $q-\kappa q-\beta/\epsilon+Q>0$ gives
$$r_{j(x)}^{p(1-\kappa)-\beta/\epsilon+Q} j(x)^{-\lambda_2}\lesssim 1.$$
Consequently, \eqref{add-eq-2} implies that
\begin{align*}
\int_\bx\Phi({\tilde f}^*(\xi))\, d\nu &\lesssim \Phi(|f(0)|) +\int_X \Phi(g_f(x)) r_{j(x)}^{q-\kappa q-\beta/\epsilon+Q}j(x)^{-\lambda_2}\, d\mull(x)\\
&\lesssim  \Phi(|f(0)|) +\int_X \Phi(g_f(x)) \, d\mull(x).
\end{align*}
Actually, the value $|f(0)|$ is not essential. For any $y\in\{x\in X: |x|<1\}$, a neighborhood of $0$, we could modify the definition of ${\tilde f}^*(\xi)$ as
$${\tilde f}^*(\xi)=|f(y)|+|f(y)-f(0)|+\sum_{j=0}^{+\infty} |f(x_{j+1})-f(x_{j})|.$$
Since $\mull(X)\approx 1$, we have that
$$\Phi(|f(y)-f(0)|)\leq\Phi\left(\int_{[0, y]} g_f\, ds\right) \leq\Phi\left(\int_{X} g_f\, ds\right)\lesssim \int_X \Phi(g_f)\, d\mull.$$
By the same argument  as above, we obtain the estimate 
$$\int_\bx\Phi({\tilde f}^*(\xi))\, d\nu(\xi)\lesssim \Phi(|f(y)|)+\int_X \Phi(g_f)\, d\mull,$$
for any $y\in \{x\in X: |x|<1\}$.  The fact that $f\in L^{\Phi}(X, \mull)$ gives us that $\Phi(|f(y)|)<\infty$ for $\mull$-a.e.  $y\in X$. This shows that ${\tilde f}^*(\xi)$ is $L^\Phi$-integrable on $\bx$, which finishes the proof of the existence of the limit in \eqref{trace-operator}. 

We continue towards norm estimates. Since $|\tilde f|\leq {\tilde f}^*$ for any modified ${\tilde f}^*$, the above arguments also show that for any $y\in \{x\in X: |x|<1\}$, we have that
$$\int_\bx\Phi(\tilde f(\xi))\, d\nu(\xi) \lesssim\Phi(|f(y)|)+ \int_X \Phi(g_f)\, d\mull.$$
Integrating over all $y\in \{x\in X: |x|<1\}$, since $\mull(\{x\in X: |x|<1\})\approx 1$, we arrive at the estimate 
\begin{equation}\label{energy}
\int_\bx\Phi(\tilde f(\xi))\, d\nu(\xi)\lesssim \int_X \Phi(|f|)\, d\mull+\int_X \Phi(g_f)\, d\mull.
\end{equation}
Assume that $\|f\|_{L^\Phi(X, \mull)}=t_1$ and $\|g_f\|_{L^\Phi(X, \mull)}=t_2$. By the definition of Luxemburg norms, we know that 
$$\int_{X}\Phi(f/t_1)\, d\mull\leq 1\ \ \ {\rm and} \ \ \ \int_{X}\Phi(g_f/t_2)\, d\mull\leq 1.$$
By estimate \eqref{energy}, there exists a constant $C>0$ such that
$$\int_\bx\Phi(\tilde f(\xi))\, d\nu(\xi)\lesssim C\left(\int_X \Phi(|f|)\, d\mull+\int_X \Phi(g_f)\, d\mull\right).$$
We may assume $C\geq 1$, since if $C<1$, we choose $C=1$. Then we obtain that
\begin{align*}
\int_\bx \Phi\left(\frac{\tilde f(\xi)}{2C(t_1+t_2)}\right)\, d\nu&\leq C\left(\int_X\Phi\left(\frac{f}{2Ct_1}\right)\, d\mull+\int_X\Phi\left(\frac{g_f}{2Ct_2}\right)\, d\mull\right)\\
&\leq \frac12\left(\int_{X}\Phi(f/t_1)\, d\mull+\int_{X}\Phi(g_f/t_2)\, d\mull\right)\leq 1,
\end{align*}
which implies
\begin{equation}\label{L_Phi}
\|\tilde f(\xi)\|_{L^\Phi(\bx)}\leq 2C(t_1+t_2)\approx \|f\|_{L^\Phi(X, \mull)}+\|g_f\|_{L^\Phi(X,\mull)}=\|f\|_{N^{1, \phi}(X, \mull)}.
\end{equation}

Next, we estimate the dyadic energy $|\tilde f|_{\dobesov}$. Given $I\in \dyadic_n$, $\xi\in I$ and $\zeta\in \widehat I$, we have $x_{n-1}=y_{n-1}$, where $x_j=x_j(\xi)$ and $y_j=y_j(\zeta)$ are the ancestors of $\xi$ and $\zeta$ with $|x_j|=|y_j|=j$, and hence that
\begin{equation}\label{add-eq-3}
|\tilde f(\xi)-\tilde f(\zeta)|\leq \sum_{j=n-1}^{+\infty}|f(x_j)-f(x_{j+1})|+\sum_{j=n-1}^{+\infty}|f(y_j)-f(y_{j+1})|.
\end{equation}
 By  \eqref{relation1} and an  argument similar to \eqref{estimate-f}, we infer from \eqref{add-eq-3} that
$$|\tilde f(\xi)-\tilde f(\zeta)|\lesssim \sum_{j=n-1}^{+\infty}r_j \dashint_{[x_{j}, x_{j+1}]} g_f \, d\mull+\sum_{j=n-1}^{+\infty}r_j \dashint_{[y_{j}, y_{j+1}]} g_f \, d\mull.$$
 It follows from the Jensen inequality that
\begin{align*}
\Psi\left(\frac{|\tilde f(\xi)-\tilde f(\zeta)|}{e^{-\epsilon n}}\right)&\lesssim {\sum_{j=n-1}^{+\infty}  r_{n-1}^{-1}r_j \dashint_{[x_{j}, x_{j+1}]} \Psi(g_f) \, d\mull}+{\sum_{j=n-1}^{+\infty}  r_{n-1}^{-1}r_j \dashint_{[y_{j}, y_{j+1}]} \Psi(g_f) \, d\mull},
\end{align*}
since we have the estimate 
$$r_{n-1}\approx e^{-\epsilon n}\approx \sum_{j=n-1}^{+\infty} r_j.$$
By using  the fact $\Phi=\Psi^q$ and the H\"older inequality if $q> 1$ (if $q=1$, the H\"older inequality is not needed), we get that
\begin{align*}
\Phi\left(\frac{|\tilde f(\xi)-\tilde f(\zeta)|}{e^{-\epsilon n}}\right)&=\Psi\left(\frac{|\tilde f(\xi)-\tilde f(\zeta)|}{e^{-\epsilon n}}\right)^q \\
&\lesssim r_{n-1}^{-q+\kappa q}\sum_{j=n-1}^{+\infty} r_j^{q-\beta/\epsilon-\kappa q} j^{-\lambda_2}\left(\int_{[x_{j}, x_{j+1}]} \Phi(g_f) \, d\mull+\int_{[y_{j}, y_{j+1}]} \Phi(g_f) \, d\mull\right).
\end{align*}
Since $\nu(I) \approx \nu(\widehat I)$ and  $\widehat I$ is the parent of $I$, it follows from Fubini's theorem  that
\begin{align}
\sum_{I\in\dyadic_n}\nu(I) &\Phi\left(\frac{|\tilde f_I-\tilde f_{\widehat I}|}{e^{-\epsilon n}}\right)\leq \sum_{I\in\dyadic_n}\nu(I) \dashint_I \dashint_{\widehat I}\Phi\left(\frac{|\tilde f(\xi)-\tilde f(\zeta)|}{e^{-\epsilon n}}\right)\, d\nu(\zeta)\, d\nu(\xi)\notag\\
&\lesssim \int_{\bx} r_{n-1}^{-q+\kappa q} \sum_{j=n-1}^{+\infty} r_j^{q-\beta/\epsilon-\kappa q}j^{-\lambda_2}\int_{[x_{j}, x_{j+1}]} \Phi(g_f) \, d\mull\, d\nu(\xi)\notag\\
&=\int_{X\cap \{|x|\geq n-1\}}\Phi(g_f) r_{n-1}^{-q+\kappa q}\int_{\bx}\sum_{j=n-1}^{+\infty} r_j^{q-\beta/\epsilon-\kappa q}j^{-\lambda_2}\chi_{[x_j, x_{j+1}]}(x)\, d\nu(\xi)\, d\mull(x).\label{add-eq-4}
\end{align}
Note again that $\chi_{[x_j, x_{j+1}]}(x)$ is nonzero only if $j\leq |x|\leq j+1$ and $x<\xi$. Recall $E(x)=\{\xi\in \bx: x<\xi\}$ and that $j(x)$ is the largest integer such that $j(x)\leq |x|$. Then $\nu(E(x))\lesssim r_{j(x)}^Q$. Hence \eqref{add-eq-4} gives
\begin{align*}
\sum_{I\in\dyadic_n}\nu(I) \Phi\left(\frac{|\tilde f_I-\tilde f_{\widehat I}|}{e^{-\epsilon n}}\right)&\lesssim \int_{X\cap \{|x|\geq n-1\}}\Phi(g_f) r_{n-1}^{-q+\kappa q} r_{j(x)}^{q-\beta/\epsilon-\kappa q}j(x)^{-\lambda_2}\nu(E(x))\, d\mull(x)\\
&\lesssim \int_{X\cap \{|x|\geq n-1\}}\Phi(g_f) r_{n-1}^{-q+\kappa q} r_{j(x)}^{q-\beta/\epsilon-\kappa q+Q}j(x)^{-\lambda_2} \, d\mull(x).
\end{align*}
  Since $e^{-\epsilon n}\approx r_{n-1}$, we conclude the estimate
\begin{align*}
|\tilde f|_{\dobesov}&\lesssim \sum_{n=1}^{+\infty} r_{n-1}^{(1-\theta)p-q+\kappa q} n^{\lambda_2}\int_{X\cap \{|x|\geq n-1\}}\Phi(g_f)  r_{j(x)}^{q-\beta/\epsilon-\kappa q+Q}j(x)^{-\lambda_2}\, d\mull(x)\\
&=\sum_{n=0}^{+\infty} r_{n}^{(1-\theta)p-q+\kappa q}(n+1)^{\lambda_2} \sum_{j=n}^{+\infty} \int_{X\cap\{j\leq|x|<j+1\}} \Phi(g_f)  r_j^{q-\beta/\epsilon-\kappa q+Q}j^{-\lambda_2}\, d\mull(x)\\
&= \sum_{j=0}^{+\infty} \int_{X\cap\{j\leq|x|<j+1\}} \Phi(g_f)  r_j^{q-\beta/\epsilon-\kappa q+Q}j^{-\lambda_2}\, d\mull(x) \left(\sum_{n=0}^{j} r_{n}^{(1-\theta)p-q+\kappa q}(n+1)^{\lambda_2}\right).
\end{align*}
Recall that $r_n=2e^{-n\epsilon}/\epsilon$ and 
$$(1-\theta)p-q+\kappa q=\kappa q-(q-(\beta-\log K)/\epsilon)=\kappa q+\beta/\epsilon-q-\log K/\epsilon<0.$$
Hence we obtain that
$$\sum_{n=0}^{j} r_{n}^{(1-\theta)p-q+\kappa q}(n+1)^{\lambda_2} \approx r_j^{\kappa q+\beta/\epsilon -q-\log K/\epsilon}(j+1)^{\lambda_2}= r_j^{\kappa q+\beta/\epsilon-q-Q}j^{\lambda_2}. $$
Therefore, our estimate above for the dyadic energy can be rewritten as 
$$|\tilde f|_{\dobesov} \lesssim\sum_{j=0}^{+\infty} \int_{X\cap\{j\leq|x|<j+1\}} \Phi(g_f) \, d\mull(x) =\int_X \Phi(g_f) \, d\mull(x).$$
By an  argument similar to the one that we used to prove \eqref{L_Phi} after getting \eqref{energy}, we have that
$$\inf\left\{k>0: |\tilde f/k|_{\dobesov}\leq 1\right\}\lesssim \|g_f\|_{L^{\Phi}(X, \mull)},$$
which together with \eqref{L_Phi} gives the norm estimate
$$\|\tilde f\|_{\obesov}\lesssim \|f\|_{N^{1, \Phi}(X, \mull)}.$$

{\bf Extension Part:} Fix $u\in \obesov$. Given $x\in X$ with $|x|=n\in \mathbb N$, set
\begin{equation}\label{extension-operator1}
\tilde u(x)=\dashint_{I_x} u\, d\nu,
 \end{equation} 
 where $I_x\in \dyadic_n$ is the set of all the points $\xi\in \bx$ such that the geodesic $[0, \xi)$ passes through $x$. 
 
 Let $y$ be a child of $x$. Then $|y|=n+1$ and $I_x$ is the mother of $I_y$. We define $\tilde u$ on the edge $[x, y]$ by setting
\begin{equation}\label{extension-operator2}
g_{\tilde u}(t)=\frac{\tilde u(y)-\tilde u(x)}{d_X(x, y)}=\frac{\epsilon (u_{I_y}-u_{I_x})}{(1-e^{-\epsilon})e^{-\epsilon n}}=\frac{\epsilon (u_{I_y}-u_{\widehat I_y})}{(1-e^{-\epsilon})e^{-\epsilon n}}
\end{equation}
 and
\begin{equation}\label{extension-operator3}
\tilde u(t)= \tilde u(x)+g_{\tilde u}(t)d_X(x, t).
\end{equation}
By repeating this procedure for all edges, we obtain an extension $\tilde u$ of $u$. Then \eqref{trace-operator} and \eqref{extension-operator1} imply that $\Tr \tilde u(\xi)=u(\xi)$ whenever $\xi\in \bx$ is a Lebesgue point of $u$. 

Simple integration shows that $|g_{\tilde u}|$ is an upper gradient of $\tilde u$. Clearly
\begin{align*}\int_{[x, y]}\Phi(|g_{\tilde u}|)\, d\mull&\approx \int_n^{n+1} \Phi\left(\frac{ |u_{I_y}-u_{\widehat I_y}|}{e^{-\epsilon (n+1)}}\right) e^{-\beta\tau}(\tau+C)^{\lambda_2}\, d\tau\\
&\approx e^{-\beta (n+1)}(\tau+1)^{\lambda_2}\Phi\left(\frac{ |u_{I_y}-u_{\widehat I_y}|}{e^{-\epsilon (n+1)}}\right).
\end{align*}
By summing over all the edges of  $X$, we conclude  that
\begin{equation}\label{tag1}
\int_{X}\Phi(|g_{\tilde u}|)\, d\mull \approx \sum_{n=1}^{+\infty}\sum_{I\in\dyadic_n} e^{-\beta n}n^{\lambda_2}\Phi\left(\frac{ |u_{I}-u_{\widehat I}|}{e^{-\epsilon n}}\right).
\end{equation}
 We have that
$$\nu(I) \approx e^{-\epsilon n Q}$$
whenever  $I\in \dyadic_n$, 
which implies that
\begin{equation}\label{tag2} 
e^{\epsilon n(\theta-1)p}\nu(I)\approx e^{-\epsilon n((\beta-\log K)/\epsilon+Q)}\approx e^{-\beta n}.
\end{equation}
The above estimates \eqref{tag1} and \eqref{tag2} give
\begin{equation}\label{extension-energy}
\int_{X}\Phi(|g_{\tilde u}|)\, d\mull \approx\sum_{n=1}^{\infty} e^{\epsilon n(\theta-1) p}n^{\lambda_2}\sum_{I\in \dyadic_n} \nu(I)\Phi\left(\frac{\left|u_{I}-u_{\widehat I}\right|}{e^{-\epsilon n}}\right)= |u|_{\dobesov}.
\end{equation}
 
 Towards the $L^\Phi$-estimate of $\tilde u$, notice that 
 \begin{equation}\label{L_p-relation}
 |\tilde u(t)|\leq |\tilde u(x)|+|g_{\tilde u}|d_X(x, y)=|\tilde u(x)|+|\tilde u(y)-\tilde u(x)|\lesssim |u_{I_x}|+|u_{I_y}|
 \end{equation}
 for any $t\in [x, y]$. Since $\mull([x, y])\approx e^{-\beta n}n^{\lambda_2}$ and $\nu(I_x)\approx \nu(I_y)\approx e^{-\epsilon n Q}$, this gives us
 $$\int_{[x, y]} \Phi(|\tilde u(t)|)\, d\mull \lesssim \mull([x, y])\big(\Phi(|u_{I_x}|)+\Phi(|u_{I_y}|)\big) \lesssim e^{-\beta n+\epsilon n Q} n^{\lambda_2}\int_{I_x} \Phi(|u|)\, d\nu.$$
  By summing over all the edges of  $X$, we arrive at
 \begin{align*}
 \int_X \Phi(|\tilde u(t)|)\, d\mull &\lesssim \sum_{n=0}^{+\infty} \sum_{I\in \dyadic_n} e^{-\beta n+\epsilon n Q} n^{\lambda_2}\int_{I} \Phi(|u|)\, d\nu\\
 &= \sum_{n=0}^{+\infty}e^{-\beta n+\epsilon n Q}n^{\lambda_2}\int_{\bx}\Phi(|u|)\, d\nu.
 \end{align*}
The sum of $e^{-\beta n+\epsilon n Q}n^{-\lambda_2}$ converges, because $\beta-\epsilon Q=\beta-\log K>0$. It follows that
\begin{equation}\label{extension-L_Phi}
\int_X \Phi(|\tilde u(t)|)\, d\mull\lesssim \int_{\bx}\Phi(|u|)\, d\nu.
\end{equation}

Applying the very same arguments that we used in proving \eqref{L_Phi} after getting \eqref{energy} to \eqref{extension-energy} and \eqref{extension-L_Phi},  we finally arrive at the desired estimate
$$\|\tilde u\|_{N^{1, \Phi}(X, \mull)}\lesssim \|u\|_{\obesov}.$$
\end{proof}

%%%%%%%%%%%%%%%%%%%%%%%%%%%%%%%%%%%%%%%%%%%%%%%%%
%\newpage

%%%%%%%%%%%%%%%%%%%%%%%%%%%%%%%%%%%%%%%%%%%%%%%%%%%%%%
%\newpage
\subsection{Proof of proposition \ref{prop1.2}}
In this section, we always assume that $\Phi(t)=t^p\log^{\lambda_1}(e+t)$ with $p >1, \lambda_1\in\real$ or  $p=1, \lambda_1\geq 0$.

\begin{lem}\label{lem3.1}
Let $\lambda, \lambda_1, \lambda_2\in \real$.
Assume that $\lambda_1+\lambda_2=\lambda$. For any $f\in L^1(\bx)$, we have that 
$\|f\|_{\dlbesov}<\infty$ is equivalent to $|f|_{\dobesov}<\infty$ whenever $0<\theta<1$.
\end{lem}
\begin{proof}
When $\lambda_1=0$, then the result is obvious since $\|f\|^p_{\dlbesov}=|f|_{\dobesov}$.

When $\lambda_1>0$, first we estimate the logarithmic term from above. Since $f\in L^1(\bx)$, for any $I\in \dyadic_n$, it follows from $\nu(I)\approx \nu(\widehat I)\approx e^{-n\log K}$ that
\begin{align*}
\log^{\lambda_1}\left(e+\frac{|f_I-f_{\widehat I}|}{e^{-\epsilon n}}\right)\leq \log^{\lambda_1}\left(e+\frac{|f_I|+|f_{\widehat I}|}{e^{-\epsilon n}}\right) \lesssim \log^{\lambda_1}\left(e+\frac{\|f\|_{L^1(\bx)}}{e^{-(\epsilon +\log K) n}}\right)\leq C n^{\lambda_1}, 
\end{align*}
where $C=C(\|f\|_{L^1(\bx)}, \lambda_1, \epsilon, K)$. Hence we can estimate $|f|_{\dobesov}$ as follows:
\begin{align*}
|f|_{\dobesov}&=\sum_{n=1}^{\infty} e^{\epsilon n(\theta-1) p}n^{\lambda_2}\sum_{I\in \dyadic_n} \nu(I)\Phi\left(\frac{\left|g_{I}-g_{\widehat I}\right|}{e^{-\epsilon n}}\right)\\
&=\sum_{n=1}^{\infty} e^{\epsilon n \theta p} n^{\lambda_2} \sum_{I\in \dyadic_n} \nu(I) |f_I-f_{\widehat I}|^p\log^{\lambda_1}\left(e+\frac{|f_I-f_{\widehat I}|}{e^{-\epsilon n}}\right)\\
&\leq C \sum_{n=1}^{\infty} e^{\epsilon n \theta p} n^{\lambda_2+\lambda_1} \sum_{I\in \dyadic_n} \nu(I) |f_I-f_{\widehat I}|^p=C\|f\|^p_{\dlbesov},
\end{align*}
where $C=C(\|f\|_{L^1(\bx)}, \lambda_1, \epsilon, K)$.

In order to estimate the logarithmic term from below, for any $I\in \dyadic_n$, we define
\begin{equation}\label{chi}
\chi(n, I)=\left\{\begin{array}{cc}
1, &\  \mathrm {if} \ \  |f_{I}- f_{\widehat I}|>e^{-\epsilon n(\theta+1)/2}\\
0,  & \mathrm{ otherwise}.
\end{array}
\right.
\end{equation}
Then we have that
\begin{align*}
\|f\|^p_{\dlbesov}&=\sum_{n=1}^{\infty} e^{\epsilon n \theta p} n^{\lambda} \sum_{I\in \dyadic_n} \nu(I) |f_I-f_{\widehat I}|^p\\
&=\sum_{n=1}^{\infty} e^{\epsilon n \theta p} n^{\lambda} \sum_{I\in \dyadic_n} \nu(I) \chi(n, I)|f_I-f_{\widehat I}|^p\\
&\ \ \ \ + \sum_{n=1}^{\infty} e^{\epsilon n \theta p} n^{\lambda} \sum_{I\in \dyadic_n} \nu(I) (1-\chi(n,I))|f_I-f_{\widehat I}|^p\\
&=: P_1+P_2.
\end{align*}
If $|f_{I}- f_{\widehat I}|>e^{-\epsilon n(\theta+1)/2}$, since $\theta<1$ and $\lambda_1>0$, we obtain that
$$\log^{\lambda_1}\left(e+\frac{|f_I-f_{\widehat I}|}{e^{-\epsilon n}}\right)>\log^{\lambda_1}\left(e+e^{\epsilon n (1-\theta)/2}\right)\geq C n^{\lambda_1},$$
where $C=C(\epsilon, \theta, \lambda_1)$. Hence we have the estimate
$$P_1\leq C\sum_{n=1}^{\infty} e^{\epsilon n \theta p} n^{\lambda_2} \sum_{I\in \dyadic_n} |f_I-f_{\widehat I}|^p \log^{\lambda_1}\left(e+\frac{|f_I-f_{\widehat I}|}{e^{-\epsilon n}}\right)=C |f|_{\dobesov}.$$
For $P_2$, %since for $|f_{I}- f_{\widehat I}|\leq e^{-\epsilon n(\theta+1)/2}$, 
since $\sum_{I\in\dyadic_n}\nu(I)\approx 1$, we have that

$$P_2 \leq \sum_{n=1}^{\infty} e^{\epsilon n \theta p} n^{\lambda} \sum_{I\in \dyadic_n} \nu(I) e^{-\epsilon n p(\theta+1)/2} \approx \sum_{n=1}^{\infty} e^{\epsilon n p (\theta-1)/2} n^\lambda =C'<+\infty,$$
where $C'=C'(\theta, p, \lambda)$. Therefore, we obtain
\begin{equation}\label{comparable1}
\frac 1C |f|_{\dobesov}\leq \|f\|^p_{\dlbesov}=P_1+P_2\leq C|f|_{\dobesov} +C',
\end{equation}
where $C$ and $C'$ are constants depending only on $\epsilon, \theta, \lambda_1, \lambda, p$ znd $\|f\|_{L^1(\bx)}$.

When $\lambda_1<0$, in order to estimate the logarithmic term from above, using definition \eqref{chi}, we obtain that
\begin{align*}
|f|_{\dobesov}&=\sum_{n=1}^{\infty} e^{\epsilon n \theta p} n^{\lambda_2} \sum_{I\in \dyadic_n} \nu(I) |f_I-f_{\widehat I}|^p\log^{\lambda_1}\left(e+\frac{|f_I-f_{\widehat I}|}{e^{-\epsilon n}}\right)\\
&=\sum_{n=1}^{\infty} e^{\epsilon n \theta p} n^{\lambda_2} \sum_{I\in \dyadic_n} \nu(I) \chi(n, I)|f_I-f_{\widehat I}|^p\log^{\lambda_1}\left(e+\frac{|f_I-f_{\widehat I}|}{e^{-\epsilon n}}\right)\\
&\ \ \ \ \ \ +\sum_{n=1}^{\infty} e^{\epsilon n \theta p} n^{\lambda_2} \sum_{I\in \dyadic_n} \nu(I) (1-\chi(n, I))|f_I-f_{\widehat I}|^p\log^{\lambda_1}\left(e+\frac{|f_I-f_{\widehat I}|}{e^{-\epsilon n}}\right)\\
&=:P_1'+P_2'.
\end{align*}
If $|f_{I}- f_{\widehat I}|>e^{-\epsilon n(\theta+1)/2}$, since $\theta<1$ and $\lambda_1<0$, we have that
$$\log^{\lambda_1}\left(e+\frac{|f_I-f_{\widehat I}|}{e^{-\epsilon n}}\right)<\log^{\lambda_1}\left(e+e^{\epsilon n (1-\theta)/2}\right)\leq C n^{\lambda_1},$$
where $C=C(\epsilon, \theta, \lambda_1)$. Hence we have the estimate
$$P_1'\leq C\sum_{n=1}^{\infty} e^{\epsilon n \theta p} n^{\lambda_2+\lambda_1} \sum_{I\in \dyadic_n} \nu(I) |f_I-f_{\widehat I}|^p=C \|f\|^p_{\dlbesov}.$$
For $P_2'$, since $\log^{\lambda_1}(e+t)\leq 1$ for any $t\geq 0$ and $\sum_{I\in\dyadic_n}\nu(I)\approx 1$, we obtain that
$$P_2'\leq \sum_{n=1}^{\infty} e^{\epsilon n \theta p} n^{\lambda_2} \sum_{I\in \dyadic_n} \nu(I) e^{-\epsilon n(\theta+1)/2} =\sum_{n=1}^{\infty} e^{\epsilon n  p(\theta-1)/2} n^{\lambda_2}=C'<+\infty,$$
where $C'=C(\epsilon, \theta, \lambda_2)$. 

Next, we estimate the logarithmic term from below.  Since $f\in L^1(\bx)$ and $\lambda_1<0$, for any $I\in \dyadic_n$, it follows from $\nu(I)\approx \nu(\widehat I)\approx e^{-n\log K}$ that
\begin{align*}
\log^{\lambda_1}\left(e+\frac{|f_I-f_{\widehat I}|}{e^{-\epsilon n}}\right)\geq \log^{\lambda_1}\left(e+\frac{|f_I|+|f_{\widehat I}|}{e^{-\epsilon n}}\right) \gtrsim \log^{\lambda_1}\left(e+\frac{\|f\|_{L^1(\bx)}}{e^{-(\epsilon +\log K) n}}\right)\geq C n^{\lambda_1}, 
\end{align*}
where $C=C(\|f\|_{L^1(\bx)}, \lambda_1, \epsilon, K)$. Now we get the estimate 
\begin{align*}
\|f\|^p_{\dlbesov}&=\sum_{n=1}^{\infty} e^{\epsilon n \theta p} n^{\lambda_2+\lambda_1} \sum_{I\in \dyadic_n} \nu(I) |f_I-f_{\widehat I}|^p\\
&\leq C \sum_{n=1}^{\infty} e^{\epsilon n \theta p}n^{\lambda_2}\sum_{I\in \dyadic_n} \nu(I) |f_I-f_{\widehat I}|^p\log^{\lambda_1}\left(e+\frac{|f_I-f_{\widehat I}|}{e^{-\epsilon n}}\right)\\
&=C|f|_{\dobesov}.
\end{align*}
Therefore, we obtain the estimate
\begin{equation}\label{comparable2}
\frac 1C \|f\|^p_{\dlbesov}\leq |f|_{\dobesov}=P_1'+P_2'\leq C\|f\|^p_{\dlbesov}+C',
\end{equation}
where $C$ and $C'$ are constants depending only on $\epsilon, \theta, \lambda_1, \lambda_2$ and $\|f\|_{L^1(\bx)}$.

Combining the inequalities \eqref{comparable1} and \eqref{comparable2} which are respect to $\lambda_1>0$ and $\lambda_1<0$ with the case $\lambda_1=0$, we obtain that $\|f\|^p_{\dlbesov}<+\infty$ is equivalent to $|f|_{\dobesov}<+\infty$.
\end{proof}

We need a result from functional ananlysis.
\begin{lem}[Closed graph theorem]\label{lem3.2}
Let $X, Y$ be Banach spaces and let $T: X\rightarrow Y$ be a linear operator. Then $T$ is continuous if and only if the graph $\sum :=\{(x, T(x)): x\in X\}$ is closed in $X\times Y$ with the product topology.
\end{lem}

 Let $L^{\Phi}(\bx)\cap \dlbesov$ be the Banach space equipped with the norm
\[\|f\|_{L^{\Phi}(\bx)\cap \dlbesov} := \|f\|_{L^{\Phi}(\bx)}+\|f\|_{\dlbesov}.\]
Using the same manner, we could define the space $X\cap Y$ for any two spaces $X$ and $Y$.
% Let $\obesov$ be the Banach space equipped with the norm
% \[\|f\|_{\obesov} := \|f\|_{L^{\Phi}(\bx)}+\inf\left\{k>0: |f/k|_{\dobesov}\leq 1\right\}.\]

\begin{cor}\label{cor3.3}
Let $\lambda, \lambda_1,\lambda_2$ and $\Phi$ be as in Lemma \ref{lem3.1}.
Then we have
\[ L^{\Phi}(\bx)\cap \dlbesov=\obesov\]
with equivalent norms.
\end{cor}
\begin{proof}
It directly follows from Lemma \ref{lem3.1} that $L^{\Phi}(\bx)\cap \dlbesov$ and $\obesov$ are the same vector spaces. Next we use Lemma \ref{lem3.2} (Closed graph theorem) to show that they are the same Banach spaces with  equivalent norms.

Consider the identity map $\Id:  L^{\Phi}(\bx)\cap \dlbesov\rightarrow \obesov$, i.e., $\Id(x)=x$ for any $x\in  L^{\Phi}(\bx)\cap \dlbesov$. Then the graph of $\Id$ is closed. Indeed, if $(x_n, x_n)$ is a sequence in this graph that converges to $(x, y)$ in $(L^{\Phi}(\bx)\cap \dlbesov)\times (L^p(\bx)\cap \dobesov)$ with product topology, then $x_n$ converges to $x$ in $\|\cdot\|_{L^{\Phi}(\bx)\cap \dlbesov}$ norm and hence in $L^{\Phi}(\bx)$. In the same manner, $x_n$ converges to $y$ in $\|\cdot\|_{\obesov}$ and hence in $L^{\Phi}(\bx)$. But the limits are unique in $L^{\Phi}(\bx)$, so $x=y$. 

Applying Lemma \ref{lem3.2} (Closed graph theorem), we see that the map $\Id$ is continuous from $L^{\Phi}(\bx)\cap \dlbesov$ to $\obesov$; similarly for  the inverse. Thus the norms $\|\cdot\|_{L^{\Phi}(\bx)\cap \dlbesov}$ and $\|\cdot\|_{\obesov}$ are equivalent and the claim follows.
\end{proof}

There is a slightly difference between the results in Corollary \ref{cor3.3} and Proposition \ref{prop1.2}, since   $\lbesov=L^p(\bx)\cap \dlbesov$. To get Proposition \ref{prop1.2} from Corollary \ref{cor3.3}, we need some estimates between the $L^p$-norm and $L^\Phi$-norm. Since $\nu(\bx)=1$, we have the following lemma, see \cite [Theorem 3.17.1 and Theorem 3.17.5]{KJF}.

\begin{lem}
Let $\Phi_1, \Phi_2$ be two Young functions. 
If  $\Phi_2\prec\Phi_1$,  then
$$\|u\|_{L^{\Phi_2}(\bx)}\lesssim \|u\|_{L^{\Phi_1}(\bx)}$$
for all $u\in L^{\Phi_1}(\bx)$.
\end{lem}

By the relation \eqref{eq-add}, for any $\delta>0$,  we have 
\begin{equation}\label{eq-add1}
\|u\|_{L^{\max\{p-\delta, 1\}}(\bx)}\lesssim \|u\|_{L^{\Phi}(\bx)}\lesssim \|u\|_{L^{p+\delta}(\bx)}
\end{equation}
  for all $u\in L^{p+\delta}(\bx)$.

Recall that $\nu(\bx)=1$ and $\diam(\bx)\approx 1$. Since $\bx$ is Ahlfors $Q$-regular where $Q=\frac{\log K}{\epsilon}$, we obtain the following lemma immediately from \cite[Theorem 4.2]{K19}
\begin{lem}\label{lemma3.6}
Let $0<s<1$ and $p\geq 1$. Let $u\in \dot N^{s}_{p,p}(\bx)$. If $0<sp<Q=\frac{\log K}{\epsilon}$, then $u\in L^{p^*}(\bx)$, $p^*=\frac{Qp}{Q-sp}$ and 
\[\inf_{c\in \mathbb R}\left(\dashint_{\bx}|u-c|^{p^*}\, d\nu\right)^{1/{p^*}}\lesssim \|u\|_{\dot N^{s}_{p, p}(\bx)}\]
\end{lem}

\begin{proof}[Proof of Proposition \ref{prop1.2}]
Let $s=\min\{\frac{\theta}{2}, \frac{Q}{2p}\}$, where $Q=\frac{\log K}{\epsilon}$. Let $p^*=\frac{Qp}{Q-sp}$ and $\delta=p^*-p$. It follows from the definitions of our Besov-type spaces and Proposition \ref{norm-equiv} that
$$\dlbesov\subset \dot {\mathcal B}^{s}_{p}(\bx)=\dot N^{s}_{p, p}(\bx).$$
By Lemma \ref{lemma3.6} and triangle inequality, we obtain that
\begin{align*}
\left(\dashint_{\bx}|u-u_{\bx}|^{p^*}\, d\nu\right)^{1/{p^*}}&\leq 2 \inf_{c\in \mathbb R}\left(\dashint_{\bx}|u-c|^{p^*}\, d\nu\right)^{1/{p^*}}\\
&\lesssim \|u\|_{\dot N^{s}_{p, p}(\bx)}\lesssim \|u\|_{\dlbesov},
\end{align*}
for any $u\in \dlbesov$, where $u_{\bx}=\dashint_{\bx} u\, d\nu$. Since $|u|\leq |u-u_{\bx}|+|u_{\bx}|$ and $\nu(\bx)=1$, it follows from the Minkowski inequality that
\begin{align*}
\|u\|_{L^{p^*}(\bx)}&\leq \|u-u_{\bx}\|_{L^{p^*}(\bx)}+ \|u_{\bx}\|_{L^{p^*}(\bx)}\\
&=\left(\dashint_{\bx}|u-u_{\bx}|^{p^*}\, d\nu\right)^{1/{p^*}}+\left|\dashint_{\bx} u\, d\nu\right|\\
&\lesssim \|u\|_{L^1(\bx)}+ \|u\|_{\dlbesov},
\end{align*}
for any $u\in \dlbesov$. Since $\|\cdot\|_{L^1(\bx)}\leq \|\cdot\|_{L^p(\bx)}\leq \|\cdot\|_{L^{p^*}(\bx)}$ is trivial, we have that
\[L^1(\bx) \cap \dlbesov=\lbesov= L^{p^*}(\bx)\cap \dlbesov.\] 
Recall the relation \eqref{eq-add1} and $\delta=p^*-p$.  Hence we have that
\[\|\cdot\|_{L^1}(\bx) \lesssim \|\cdot\|_{L^{\Phi}(\bx)}\lesssim \|\cdot\|_{L^{p^*}(\bx)}.\]
Thus,
\[\lbesov= L^\Phi(\bx)\cap \dlbesov.\]
Combining with Corollary \ref{cor3.3}, i.e., 
\[L^\Phi(\bx)\cap \dlbesov=L^\Phi(\bx)\cap \dobesov=\obesov,\]
we finally arrive at
\[\lbesov=\obesov.\]
\end{proof}

\bigskip

\noindent{\bf Acknowledgement.} The author would like to thank his advisor Professor Pekka Koskela for helpful comments and suggestions.

%\newpage

\medskip
\noindent Zhuang Wang

\noindent
Department of Mathematics and Statistics, University of Jyv\"askyl\"a, PO~Box~35, FI-40014 Jyv\"askyl\"a, Finland

\noindent{\it E-mail address}:  \texttt{zhuang.z.wang@jyu.fi}

\end{document}